\documentclass[12pt]{amsart}

\usepackage{amsmath,amssymb,amscd,amsthm}
\usepackage{stmaryrd}
\usepackage{hyperref}
\usepackage[left=3cm,right=3cm,top=3cm,bottom=3cm,includeheadfoot]{geometry}

\makeatletter
\@namedef{subjclassname@2010}{%
  \textup{2010} Mathematics Subject Classification}
\makeatother

\newtheorem{theorem}{Theorem}

\newtheorem {corollary}{Corollary}

\theoremstyle{definition}

\newtheorem{remark}{Remark}

\DeclareMathOperator\Sym{Sym}
\DeclareMathOperator\defi{def.}
\DeclareMathOperator\Hom{Hom}
\DeclareMathOperator\pt{point}

\begin{document}

\title[]{Identities involving (doubly) symmetric polynomials and integrals over Grassmannians}

\author[]{Dang Tuan Hiep}

\address{Mathematics Division, National Center for Theoretical Sciences, No. 1 Sec. 4 Roosevelt Rd., National Taiwan University, Taipei, 106, Taiwan}
\email{hdang@ncts.ntu.edu.tw}

\curraddr{Faculty of Mathematics and Computer Science, Da Lat University, \\ No. 1 Phu Dong Thien Vuong Rd., Ward 8, Da Lat City, Vietnam}
\email{hiepdt@dlu.edu.vn}


\keywords{Interpolation, (doubly) symmetric polynomial, localization, equivariant cohomology, Grassmannian.\\
This research was supported by grants no. B2018.DNA.10 and no. 101.04-2018.305.}

\begin{abstract}
We obtain identities involving symmetric and doubly symmetric polynomials. These identities provide a way of handling expressions appearing in the Atiyah-Bott-Berline-Vergne formula for Grassmannians. As corollaries, we obtain formulas for integrals over Grassmannians of characteristic classes of the tautological bundles. Moreover, we provide a valid proof of the Martin formula for the classical Grassmannian.
\end{abstract}

\maketitle

\section{Introduction}

Throughout we always assume that all polynomials are over a field and $\lambda_1,\ldots,\lambda_n$ are indeterminates. The Lagrange interpolation formula says that a polynomial $P(x)$ of degree not greater than $n-1$ in one variable can be written in the following
\begin{equation}\label{interpolation}
    P(x) = \sum_{i=1}^nP(\lambda_i)L_i(x),
\end{equation}
where 
$$L_i(x) = \prod_{j\neq i}\frac{x-\lambda_j}{\lambda_i-\lambda_j}.$$
This implies the following identity
\begin{equation}\label{1}
    \sum_{i=1}^n\frac{P(\lambda_i)}{\displaystyle\prod_{j\neq i}(\lambda_i-\lambda_j)}= c_n,
\end{equation}
where $c_n$ is the coefficient of $x^{n-1}$ in the polynomial $P(x)$.

The first goal of this paper is to generalize the identity (\ref{1}) for multivariate symmetric polynomials. For convenience, we shall write $[n]$ for the set $\{1,2,\ldots,n\}$. For each subset $I = \{i_1,\ldots,i_k\} \subset [n]$, we denote by $\lambda_I = (\lambda_{i_1},\ldots,\lambda_{i_k})$ and $I^c = [n]\setminus I$. Recall that a polynomial $P(x_1,\ldots,x_k)$ is said to be symmetric if it is invariant under permutations of $x_1,\ldots,x_k$. We obtain the following result.

\begin{theorem}\label{main}
Let $P(x_1,\ldots,x_k)$ be a symmetric polynomial of degree not greater than $k(n-k)$ in $k$ variables $(k < n)$. Then we have the following identity
$$\sum_{I\subset [n],|I|=k}\frac{P(\lambda_I)}{\displaystyle\prod_{i\in I,j\in I^c}(\lambda_i-\lambda_j)} = \frac{c(k,n)}{k!},$$
where $c(k,n)$ is the coefficient of $x_1^{n-1}\ldots x_k^{n-1}$ in the polynomial 
$$P(x_1,\ldots,x_k)\prod_{j \neq i}(x_i-x_j).$$ 
\end{theorem}

More generally, we also obtain an identity involving doubly symmetric polynomials. %We shall use $X$ to denote the set $\{x_1,\ldots,x_k\}$ or the $k$-tuple $(x_1,\ldots,x_k)$, and $Y$ to denote the set $\{y_1,\ldots,y_{n-k}\}$ or the $(n-k)$-tuple $(y_1,\ldots,y_{n-k})$. To shorten the notation, we denote by $$Y-X = \prod_{i=1}^{n-k}\prod_{j=1}^k(y_i - x_j).$$
Recall that a polynomial $P(x_1,\ldots,x_k,y_1,\ldots,y_{n-k})$ is said to be doubly symmetric if it is invariant under permutations of $x_1,\ldots,x_k$ and permutations of $y_1,\ldots,y_{n-k}$ respectively. We obtain the following result.

\begin{theorem}\label{thm3}
Let $P(x_1,\ldots,x_k,y_1,\ldots,y_{n-k})$ be a doubly symmetric polynomial of degree not greater than $k(n-k)$. Then we have the following identity
$$\sum_{I\subset [n],|I|=k}\frac{P(\lambda_I,\lambda_{I^c})}{\displaystyle\prod_{i\in I,j\in I^c}(\lambda_i-\lambda_j)} = \frac{d(k,n)}{k!(n-k)!},$$
where $d(k,n)$ is the coefficient of $x_1^{n-1}\ldots x_k^{n-1}y_1^{n-1}\ldots y_{n-k}^{n-1}$ in the polynomial
$$P(x_1,\ldots,x_k,y_1,\ldots,y_{n-k})\prod_{j \neq i}(x_i-x_j)\prod_{j \neq i}(y_i-y_j)\prod_{i=1}^{n-k}\prod_{j=1}^k(y_i - x_j).$$
\end{theorem}

The second goal of this paper is to give a way of dealing with integrals over Grassmannians. The idea is as follows. Localization in equivariant cohomology allows us to express integrals in terms of some data attached to the fixed points of a torus action. In particular, for Grassmannians, we obtain interesting formulas with nontrivial relations involving rational functions. Let $G(k,n)$ be the Grassmannian of $k$-dimensional linear spaces in $\mathbb C^n$. Consider the following integrals:
$$\int_{G(k,n)}\Phi(\mathcal S)\quad ,\quad \int_{G(k,n)}\Psi(\mathcal Q)\quad,\quad \int_{G(k,n)}\Delta(\mathcal S,\mathcal Q),$$
where $\Phi(\mathcal S), \Psi(\mathcal Q)$ are respectively characteristic classes of the tautological sub-bundle $\mathcal S$ and quotient bundle $\mathcal Q$ on the Grassmannian $G(k,n)$, and $\Delta(\mathcal S,\mathcal Q)$ is a characteristic class of both $\mathcal S$ and $\mathcal Q$.

Using localization in equivariant cohomology, Weber \cite{We} and Zielenkiewicz \cite{Zi} presented a way of expressing the integrals as iterated residues at infinity of holomorphic functions. Our identities provide another method for dealing with such expressions.

\begin{corollary}\label{integral1}{\rm (Compare with \cite[Formula (4)]{We})}
Suppose that $\Phi(\mathcal S)$ is represented by a symmetric polynomial $P(x_1,\ldots,x_k)$ of degree not greater than $k(n-k)$ in $k$ variables $x_1,\ldots,x_k$ which are the Chern roots of $\mathcal S$ and $\Psi(\mathcal Q)$ is represented by a symmetric polynomial $Q(y_1,\ldots,y_{n-k})$ of degree not greater than $k(n-k)$ in $n-k$ variables $y_1,\ldots,y_{n-k}$ which are the Chern roots of $\mathcal Q$. We then have the following statements:
\begin{enumerate}
\item[(a)] The integral 
$$\int_{G(k,n)}\Phi(\mathcal S) = (-1)^{k(n-k)}\frac{c(k,n)}{k!},$$
where $c(k,n)$ is the coefficient of $x_1^{n-1}\ldots x_k^{n-1}$ in the polynomial 
$$P(x_1,\ldots,x_k)\prod_{j \neq i}(x_i-x_j).$$
\item[(b)] The integral 
$$\int_{G(k,n)}\Psi(\mathcal Q) = \frac{c(k,n)}{(n-k)!},$$
where $c(k,n)$ is the coefficient of $y_1^{n-1}\ldots y_{n-k}^{n-1}$ in the polynomial 
$$Q(y_1,\ldots,y_{n-k})\prod_{j \neq i}(y_i-y_j).$$
\end{enumerate}
\end{corollary}

\begin{corollary}\label{integral2} {\rm (Compare with \cite[Formula 1]{Zi})}
Suppose that $\Delta(\mathcal S,\mathcal Q)$ is represented by a doubly symmetric polynomial $P(x_1,\ldots,x_k,y_1,\ldots,y_{n-k})$ of degree not greater than $k(n-k)$ in $n$ variables which are the Chern roots of $\mathcal S$ and $\mathcal Q$ respectively. We then have the integral
$$\int_{G(k,n)}\Delta(\mathcal S,\mathcal Q) = (-1)^{k(n-k)}\frac{d(k,n)}{k!(n-k)!},$$
where $d(k,n)$ is the coefficient of $x_1^{n-1}\ldots x_k^{n-1}y_1^{n-1}\ldots y_{n-k}^{n-1}$ in the polynomial
$$P(x_1,\ldots,x_k,y_1,\ldots,y_{n-k})\prod_{j \neq i}(x_i-x_j)\prod_{j \neq i}(y_i-y_j)\prod_{i=1}^{n-k}\prod_{j=1}^k(y_i - x_j).$$
\end{corollary}

The results of Corollary \ref{integral1} and Corollary \ref{integral2} are special cases of the formulas derived in \cite{We, Zi}. However, our proofs are new and unrelated. The formulas in the Corollaries are derived from the Theorems \ref{main} and \ref{thm3}, which are proven in a purely algebraic fashion, using the division algorithm for multivariate polynomials and iterated Lagrange interpolation. Indeed the main motivation for these results is to establish a relationship between the Atiyah-Bott-Berline-Vergne formula and Lagrange interpolation, which is a novelty.

The third goal of this paper is devoted to the Martin formula for symplectic quotients. In an unpublished paper, Martin \cite[Theorem B]{Mar} proved an integration formula, which expresses integrals on the symplectic quotient $X\sslash G$ of a Hamiltonian $G$-manifold $X$ in terms of those on the associated symplectic quotient $X\sslash T$, where $T\subset G$ is a maximal torus, in the case in which $X\sslash G$ is a compact manifold. For the case of the classical Grassmannian, the Martin formula is simplified to be the statement (a) of Corollary \ref{integral1}. This will be clarified in the final section of this paper. Therefore, the proof of Corollary \ref{integral1}, which is purely algebraic, is also an algebraic proof of the Martin formula for the classical Grassmannian. 

The rest of the paper is organized as follows: The proof of the identities are presented in Section 2. Section 3 is to give a brief review of the localization formula in equivariant cohomology and the proof of the corollaries. Section 4 is devoted to Martin's formula.
 
\section{Proof of the identities}

Set 
$$F(x_1,\ldots,x_k) = P(x_1,\ldots,x_k)\prod_{j\neq i}(x_i-x_j).$$
By the assumption, the degree of $F$ is not greater than 
$$k(n-k) + k(k-1) = k(n-1).$$
Since $P$ is symmetric, so $F$ is also symmetric. By the division algorithm for multivariate polynomials (see \cite[Theorem 3]{CLO}), there exist the polynomials $F_i(x_1,\ldots,x_k), i = 1, \ldots,k$ and $R(x_1,\ldots,x_k)$ such that 
$$R(x_1,\ldots,x_k) = F(x_1,\ldots,x_k) - \sum_{i=1}^kF_i(x_1,\ldots,x_k)\prod_{j=1}^n(x_i-\lambda_j),$$
and all partial degrees of $R$ are not greater that $n-1$.
By the Lagrange interpolation formula, we have
$$R(x_1,\ldots,x_k) =  \sum_{i_1=1}^nR(\lambda_{i_1},x_2,\ldots,x_k)L_{i_1}(x_1).$$
By the Lagrange interpolation formula for the polynomials $R(\lambda_{i_1},x_2,\ldots,x_k)$, we have
$$R(x_1,\ldots,x_k) = \sum_{i_1=1}^n\sum_{i_2=1}^nR(\lambda_{i_1},\lambda_{i_2},x_3,\ldots,x_k)L_{i_1}(x_1)L_{i_2}(x_2).$$
So on, we have
$$R(x_1,\ldots,x_k) = \sum_{i_1,\ldots,i_k=1}^nR(\lambda_I)\prod_{l=1}^kL_{i_l}(x_l).$$
For each $I=\{i_1,\ldots,i_k\}$, we have $R(\lambda_I) = F(\lambda_I)$, and if $i_s = i_t$ for some $s\neq t$, then $R(\lambda_I)=0$. Since the degree of $F$ is not greater than $k(n-1)$, so the coefficient of $x_1^{n-1}\ldots x_k^{n-1}$ in $R$ is equal to that in $F$. Thus the coefficient of $x_1^{n-1}\ldots x_k^{n-1}$ in $F$ is equal to
$$k!\sum_{I\subset[n],|I|=k}\frac{F(\lambda_I)}{\displaystyle\prod_{i\in I,j\neq i}(\lambda_i-\lambda_j)}.$$
For each $I \subset [n]$, we have 
$$F(\lambda_I) = P(\lambda_I)\prod_{i,j\in I,j\neq i}(\lambda_i-\lambda_j),$$
and
$$\prod_{i\in I,j\neq i}(\lambda_i-\lambda_j) = \prod_{i\in I,j\in I^c}(\lambda_i-\lambda_j)\prod_{i,j\in I,j\neq i}(\lambda_i-\lambda_j).$$
This implies that the coefficient of $x_1^{n-1}\ldots x_k^{n-1}$ in $F$ is equal to 
$$k!\sum_{I\subset [n],|I|=k}\frac{P(\lambda_I)}{\displaystyle\prod_{i\in I,j\in I^c}(\lambda_i-\lambda_j)}.$$
Theorem \ref{main} is proved as desired. 

The proof of Theorem \ref{thm3} is very similar to that of Theorem \ref{main}. It is omitted here.

\begin{remark}
If $P(x_1,\ldots,x_k)$ is a symmetric polynomial whose partial degrees are not greater than $n-k$, then we have the following formula, which was proved by Chen and Louck \cite[Theorem 2.1]{CL},
$$P(x_1,\ldots,x_k) = \sum_{I\subset [n],|I|=k}P(\lambda_I)\frac{\displaystyle\prod_{x\in X,j\in I^c}(x-\lambda_j)}{\displaystyle\prod_{i\in I,j\in I^c}(\lambda_i-\lambda_j)}.$$
By the interpolation formula of Chen and Louck, we get
$$\sum_{I\subset [n],|I|=k}\frac{P(\lambda_I)}{\displaystyle\prod_{i\in I,j\in I^c}(\lambda_i-\lambda_j)} = d(k,n),$$
where $d(k,n)$ is the coefficient of $x_1^{n-k}\ldots x_k^{n-k}$ in $P$. This is in fact a special case of Theorem \ref{main}. It was proved in \cite{Ze} that $k!$ is the coefficient of $x_1^{k-1}\ldots x_k^{k-1}$ in the polynomial 
$$\prod_{j \neq i}(x_i-x_j).$$
If the partial degrees of $P$ are not greater than $n-k$, then we get
$$c(k,n) = d(k,n)k!.$$
\end{remark}

\begin{remark}
Theorem \ref{thm3} is a generalization of Theorem \ref{main}. Indeed, if $P(x_1, \ldots ,x_k)$ is a symmetric polynomial, then it is also a doubly symmetric polynomial. Theorem \ref{thm3} says that the sum
$$\sum_{I\subset [n],|I|=k}\frac{P(\lambda_I)}{\displaystyle\prod_{i\in I,j\in I^c}(\lambda_i-\lambda_j)} = \frac{d(k,n)}{k!(n-k)!},$$
where $d(k,n)$ is the coefficient of $x_1^{n-1}\ldots x_k^{n-1}y_1^{n-1}\ldots y_{n-k}^{n-1}$ in the polynomial
$$P(x_1,\ldots,x_k)\prod_{j \neq i}(x_i-x_j)\prod_{j \neq i}(y_i-y_j)\prod_{i=1}^{n-k}\prod_{j=1}^k(y_i - x_j).$$
It was proved in \cite{Ze} that $(n-k)!$ is the coefficient of $y_1^{n-k-1}\ldots y_{n-k}^{n-k-1}$ in the polynomial
$$\prod_{j \neq i}(y_i-y_j).$$
Thus $(n-k)!$ is also the coefficient of $y_1^{n-1}\ldots y_{n-k}^{n-1}$ in the polynomial
$$\prod_{j \neq i}(y_i-y_j)\prod_{i=1}^{n-k}\prod_{j=1}^k(y_i - x_j).$$
This means that
$$\frac{d(k,n)}{(n-k)!} = c(k,n),$$
which is the coefficient of $x_1^{n-1}\ldots x_k^{n-1}$ in the polynomial 
$$P(x_1,\ldots,x_k)\prod_{j \neq i}(x_i-x_j).$$
The statement of Theorem \ref{main} is obtained as desired.
\end{remark}

\section{Localization in equivariant cohomology}

In this section, we recall some basic definitions and results in the theory of equivariant cohomology. For more details on this theory, we refer to \cite{AB, BV, Bo, Br, CK, EG1, EG2}. Throughout we consider all cohomologies with coefficients in the complex field $\mathbb C$.

Let $T = (\mathbb C^*)^n$ be an algebraic torus of dimension $n$, classified by the principal $T$-bundle $ET \to BT$, whose total space $ET$ is contractible. Let $X$ be a compact space endowed with a $T$-action. Put $X_T = X\times_T ET$, which is itself a bundle over $BT$ with fiber $X$. Recall that the $T$-equivariant cohomology of $X$ is defined to be $H^*_{T}(X) = H^*(X_T)$, where $H^*(X_T)$ is the ordinary cohomology of $X_T$. Note that $H^*_T(\pt) = H^*(BT)$. By pullback via the map $X\to \pt$, we see that $H^*_T(X)$ is an $H^*(BT)$-module. Thus we may consider $H^*(BT)$ as the coefficient ring for equivariant cohomology.
 
A $T$-equivariant vector bundle is a vector bundle $E$ on $X$ together with a lifting of the action on $X$ to an action on $E$ which is linear on fibers. Note that $E_T$ is a vector bundle over $X_T$.

The $T$-equivariant Chern classes $c_i^T(E) \in H^*_T(X)$ are defined to be the Chern classes $c_i(E_T)$. If $E$ has rank $r$, then the top Chern class $c_r^T(E)$ is called the $T$-equivariant Euler class of $E$ and is denoted $e^T(E) \in H^*_T(X)$. More generally, the $T$-equivariant characteristic class $c^T(E) \in H^*_T(X)$ is defined to be the characteristic class $c(E_T)$.

Let $\chi(T)$ be the character group of the torus $T$. For each $\rho \in \chi(T)$, let $\mathbb C_{\rho}$ denote the one-dimensional representation of $T$ determined by $\rho$. Then $L_{\rho} = (\mathbb C_{\rho})_T$ is a line bundle over $BT$, and the assignment $\rho \mapsto -c_1(L_{\rho})$ defines an group isomorphism $f: \chi(T) \simeq H^2(BT)$, which induces a ring isomorphism $\Sym(\chi(T)) \simeq H^*(BT)$. We call $f(\rho)$ the weight of $\rho$. In particular, we denote by $\lambda_i$ the weight of $\rho_i$ defined by $\rho_i(x_1,\ldots,x_n) = x_i$. We thus obtain an isomorphism
$$H^*_T(\pt) = H^*(BT) \simeq \mathbb C[\lambda_1,\ldots,\lambda_n].$$
Let $\mathcal R_T \simeq \mathbb C(\lambda_1,\ldots,\lambda_n)$ be the field of fractions of $\mathbb C[\lambda_1,\ldots,\lambda_n]$. An important result in equivariant cohomology is the localization theorem. Historically, localization in equivariant cohomology was studied by Borel \cite{Bo} and then further investigated by Quillen \cite{Q}, Atiyah-Bott \cite{AB}, and Berline-Vergne \cite{BV}. Among many versions of the formulation of the localization theorem, we choose the one by Atiyah and Bott \cite{AB}.

\begin{theorem} [Atiyah-Bott \cite{AB}]
Let $X^T$ be the fixed point locus of the torus action. Then the inclusion $i : X^T \hookrightarrow X$ induces an isomorphism
$$i^* : H^*_T(X) \otimes \mathcal R_T \simeq H^*_T(X^T)\otimes \mathcal R_T.$$
\end{theorem}

Moreover, Atiyah and Bott \cite{AB} also gave an explicit formula for the inverse isomorphism. If $X$ is a compact manifold and $X^T$ is finite, then the localization theorem can be rephrased as follows:

\begin{theorem} [Atiyah-Bott \cite{AB}, Berline-Vergne \cite{BV}]
Suppose that $X$ is a compact manifold endowed with a torus action and the fixed point locus $X^T$ is finite. For $\alpha \in H^*_T(X)$, we have
\begin{equation}\label{ABBV}
\int_X\alpha = \sum_{p\in X^T}\frac{\alpha|_p}{e_p},
\end{equation}
where $e_p$ is the $T$-equivariant Euler class of the tangent bundle at the fixed point $p$, and $\alpha|_p$ is the restriction of $\alpha$ to the point $p$.
\end{theorem}

For many applications, the Atiyah-Bott-Berline-Vergne formula can be formulated in more down-to-earth terms. We are mainly interested in the computation of integrals over Grassmannians.

\begin{proof}[Proof of Corollary \ref{integral1}]
Consider the action of $T = (\mathbb C^*)^n$ on $\mathbb C^n$ given in coordinates by
$$(a_1,\ldots,a_n) \cdot (x_1, \ldots , x_n) = (a_1x_1, \ldots , a_nx_n).$$
This induces a torus action on the Grassmannian $G(k,n)$ with isolated fixed points $p_I$ corresponding to coordinate $k$-planes in $\mathbb C^n$. Each fixed point $p_I$ is indexed by a subset $I\subset [n]$ of size $k$. By the Atiyah-Bott-Berline-Vergne formula, we have
$$\int_{G(k,n)}\Phi(\mathcal S) = \sum_{p_I}\frac{\Phi^T(\mathcal S|_{p_I})}{e_{p_I}}$$
and 
$$\int_{G(k,n)}\Psi(\mathcal Q) = \sum_{p_I}\frac{\Psi^T(\mathcal Q|_{p_I})}{e_{p_I}}.$$
For each $p_I$, the torus actions on the fibers $\mathcal S|_{p_I}$ and $\mathcal Q|_{p_I}$ have the characters $\rho_i$ for $i\in I$ and $\rho_j$ for $j\in I^c$ respectively. Combining with the assumption, these imply that the $T$-equivariant characteristic classes at $p_I$
$$\Phi^T(\mathcal S|_{p_I}) = P(\lambda_I)$$
and
$$\Psi^T(\mathcal Q|_{p_I}) = Q(\lambda_{I^c}).$$
Since the tangent bundle is isomorphic to $\mathcal S^\vee \otimes \mathcal Q$, the characters of the torus action on the tangent bundle at $p_I$ are
$$\{\rho_i^{-1}\rho_j \mid i \in I, j \in I^c \}.$$
Thus the $T$-equivariant Euler class of the tangent bundle at $p_I$ is 
$$e_{p_I} = \prod_{i\in I,j\in I^c}(\lambda_j-\lambda_i) = (-1)^{k(n-k)}\prod_{i\in I,j\in I^c}(\lambda_i-\lambda_j).$$
Therefore, we obtain 
$$\int_{G(k,n)}\Phi(\mathcal S) = (-1)^{k(n-k)}\sum_{I\subset[n],|I|=k}\frac{P(\lambda_I)}{\displaystyle\prod_{i\in I,j\in I^c}(\lambda_i-\lambda_j)}$$
and
$$\int_{G(k,n)}\Psi(\mathcal Q) = \sum_{I\subset[n],|I|=k}\frac{Q(\lambda_{I^c})}{\displaystyle\prod_{i\in I,j\in I^c}(\lambda_j-\lambda_i)} = \sum_{I\subset[n],|I|=n-k}\frac{Q(\lambda_{I})}{\displaystyle\prod_{i\in I,j\in I^c}(\lambda_i-\lambda_j)}.$$
Combining with Theorem \ref{main}, Corollary \ref{integral1} is proved as desired.
\end{proof}
The proof of Corollary \ref{integral2} is very similar to that of Corollary \ref{integral1}. It is omitted here.

\section{Martin's formula}

In this section, we present the Martin formula for symplectic quotients, which expresses integrals on a symplectic quotient in term of that on its associated symplectic quotient. In the case of the classical Grassmannian, we show that the statement (a) of Corollary 1 can be deduced from the Martin formula. For the more details of the construction, we refer to \cite[Section 7]{Mar}. 

Let $X$ be a compact manifold endowed with a $G$-action, where $G$ is a reductive algebraic group with a maximal torus $T \subset G$. Assume that there exist the moment maps $\mu_G$ and $\mu_T$ for $G-$ and $T-$actions on $X$ respectively. The symplectic quotients $X\sslash G$ and $X\sslash T$ are defined to be the topological quotients $\mu_G^{-1}(0)/G$ and $\mu_T^{-1}(0)/T$ respectively. Moreover, both $\mu_G^{-1}(0)$ and $\mu_T^{-1}(0)$ are assumed to be compact manifolds, on which  the respective $G-$ and $T-$actions are free. This follows that $X\sslash G$ and $X\sslash T$ are compact manifolds. There are natural inclusion $i : \mu_G^{-1}(0)/T \hookrightarrow \mu_T^{-1}(0)/T$ and projection $\pi :  \mu_G^{-1}(0)/T \twoheadrightarrow  \mu_G^{-1}(0)/G$. We say $\tilde{a}\in H^*(X\sslash T)$ is a lift of $a \in H^*(X\sslash G)$ if $\pi^*a = i^*\tilde{a}$. The set of roots is denoted by $\triangle$. We denote the line bundle on $X\sslash T$ associated with $\alpha \in \triangle$ by $L_\alpha$, and set
$$e = \prod_{\alpha\in\triangle}c_1(L_\alpha).$$
Given a cohomology class $a \in H^*(X\sslash G)$ with lift $\tilde{a} \in H^*(X\sslash T)$, it was proved by Martin in \cite[Theorem B]{Mar} that 
\begin{equation}\label{Martin}
\int_{X\sslash G}a = \frac{1}{|W|}\int_{X\sslash T} \tilde{a}\cup e,
\end{equation}
where $|W|$ is the order of the Weyl group of $G$.

The Grassmannian $G(k,n)$ can be described as the symplectic quotient of the set of complex matrices with $n$ rows and $k$ columns by the unitary group
$$G(k,n) = \Hom(\mathbb C^k, \mathbb C^n)\sslash U(k).$$
The associated symplectic quotient by the maximal torus $T \subset U(k)$ turns out to be the $k-$fold product $(\mathbb C \mathbb P^{n-1})^k$. Its cohomology ring is generated by elements $u_1,\ldots,u_k$, where $u_i$ is the positive generator of the cohomology ring of the $i$-th copy of $\mathbb C\mathbb P^{n-1}$.

The Weyl group of $U(k)$ is the symmetric group on $k$ elements $S_k$. The roots $\alpha$ of $U(k)$ can be enumerated by pairs of positive integers $(i, j)$ with $1 \leq i, j \leq k$ and $i \neq j$. The cohomology class corresponding to the root $(i, j)$ is the class $u_j - u_i$ and their product is
$$e = \prod_{i\neq j}(u_j-u_i).$$

In the last paragraph of \cite[Section 7]{Mar}, Martin described the tautological sub-bundle $\mathcal S$ on the Grassmannnian $G(k,n)$ in terms of the symplectic quotient construction as 
$$\mathcal S \cong \mu_{U(k)}^{-1}(0) \times_{U(k)} \mathbb C^k_{\defi},$$
where $\mathbb C^k_{\defi}$ denotes the defining representation of $U(k)$. Thus the dual $\mathcal S^\vee$ is constructed from the dual of the defining representation, and when we restrict this dual representation to the maximal torus, it decomposes into $k$ one-dimensional representations, which have associated line bundles on the $(\mathbb C \mathbb P^{n-1})^k$ with Euler classes $u_1, \ldots , u_k$. This implies that we can identify the Chern classes of $\mathcal S^\vee$ as elementary symmetric polynomials of the $u_i$. Using the notation and assumption as in Corollary \ref{integral1}, the lift of $\Phi(\mathcal S)$ is $(-1)^{k(n-k)}P(u_1,\ldots,u_k)$ and the Martin formula (\ref{Martin}) gives 
$$\int_{G(k,n)}\Phi(\mathcal S) = \frac{1}{k!}\int_{(\mathbb C \mathbb P^{n-1})^k}(-1)^{k(n-k)}P(u_1,\ldots,u_k)\prod_{i\neq j}(u_j-u_i).$$
It is known that the cohomology ring of $( \mathbb C \mathbb P^{n-1})^k$ is
$$H^*((\mathbb C \mathbb P^{n-1})^k) \cong \mathbb C[u_1,\ldots,u_k]/(u_1^n,\ldots,u_k^n).$$
This implies that 
$$\int_{(\mathbb C \mathbb P^{n-1})^k}P(u_1,\ldots,u_k)\prod_{i\neq j}(u_j-u_i) = c(k,n).$$
We thus have
$$\int_{G(k,n)}\Phi(\mathcal S) = (-1)^{k(n-k)} \frac{c(k,n)}{k!}.$$
It means that the statement (a) of Corollary 1 can be derived from the Martin formula together with the description of the cohomology ring of the product of projective spaces and and of the push-forward for projective spaces. Equivalently, we can start with the statement (a) of Corollary 1 and using the description of the cohomology and push-forward on the product of projective spaces arrive at the Martin formula for the classical Grassmannian. The only part of Martin's proof which does not follow from Corollary 1 is the description of the tautological subbundle on the Grassmannian, but this is a classical result. Therefore we provide a valid proof of the Martin formula for the classical Grassmannian.
\subsection*{Acknowledgements}
Part of this work was done while the author was visiting the Korea Institute for Advanced Study (KIAS). This work was finished during the author's postdoctoral fellowship at the National Center for Theoretical Sciences (NCTS), Taiwan. The author thanks all for the financial support and hospitality. The author would also like to express deep thanks to William Cherry and Phung Ho Hai for useful discussions, and to the referee for carefully reading the paper and suggestions leading to the improvement of the exposition. This research was supported by grants no. B2018.DNA.10 and no. 101.04-2018.305.

\end{document}